\documentclass{amsart}
\usepackage{amssymb}
\usepackage{amsmath}
\usepackage{latexsym}
\input{xy}
\xyoption{all}

\newtheorem{theorem}{Theorem}[section]

\newtheorem{lemma}[theorem]{Lemma}

\newtheorem*{varthm1}{Main Theorem}
\newtheorem*{eccthm}{Elliptic Carmichael Condition}
\newtheorem{Def}[theorem]{Definition}
\newtheorem*{Ko}{Korselt's Criterion}

\begin{document}
\title[Infinitely many elliptic Carmichael numbers]{There are infinitely many elliptic Carmichael numbers}
\author{Thomas Wright}
%\classification{11S45}
%\keywords{elliptic Carmichael number, Lucas-Carmichael, pseudoprime}
\thanks{Thanks to Carl Pomerance for helpful comments and for first making me aware of this problem.}
% 'Abstract comes before maketitle, as in the AMS classes'
\begin{abstract}
In 1987, Dan Gordon defined an elliptic curve analogue to Carmichael numbers known as elliptic Carmichael numbers.  In this paper, we prove that there are infinitely many elliptic Carmichael numbers.  In doing so, we resolve in the affirmative the question of whether there exist infinitely square-free, composite integers $n$ such that for every prime $p$ that divides $n$, $p+1|n+1$.
\end{abstract}
\maketitle

%\classno{Primary 11S45; Secondary 14G05}
%\maketitle

\section{Introduction}

One of the first primality tests in modern number theory came from Fermat, whose Little Theorem (if $p$ is prime then $a^p\equiv a$ (mod $p$) for any integer $a$) allows us to quickly prove compositeness for many integers $n$ by showing that $a^n\not\equiv a$ (mod $n$) for some integer $a$.  Unfortunately, the converse of Fermat's Little Theorem is not true, as there are many composite numbers $n$ for which $a^n\equiv a$ (mod $n$) for every $a$.  In particular, we have the following definition:

\begin{Def} Let $n \in \mathbb N$.  If $n|a^n-a$
for every $a \in \mathbb Z$ and $n$ is not prime then $n$ is a
\textbf{Carmichael Number}.\end{Def}
The set of Carmichael numbers was proven to be infinite by Alford, Granville, and Pomerance in 1994 \cite{AGP}.  Despite this, the set of Carmichael numbers is known to have density zero within the union of Carmichael numbers and primes, which means that although this particular test is not deterministic, it is nevertheless a reasonable way to find probable primes.

Of course, the search for Carmichael numbers depends heavily on an observation of Korselt \cite{Ko} who, in 1899, noted the following:

\begin{Ko} \textit{$n$ is a Carmichael number if and only if $n$ is squarefree and $p-1|n-1$.}\end{Ko}

In addition to laying the groundwork for the study of Carmichael numbers, Korselt's Criterion has also served to inspire other ideas for how one might construct a primality test.  In this paper we consider a primality test, created with this blueprint, that involves elliptic curves.

\section{Introduction: Elliptic Curve Carmichael Numbers}
In 1987, Dan Gordon \cite{Go1} used the ideas outlined above to devise a primality test that is most aptly stated in terms of the arithmetic of elliptic curves.

Without completely rehashing the theory of elliptic curves, we recall the following definitions.  Let $E$ be an elliptic curve over $\mathbb Q$ given by a Weierstrass equation
$$E : Y^2= X^3+ aX + b$$
with nonzero discriminant

$$\Delta = 4a^3 + 27b^2.$$
We say that $E$ has complex multiplication (or CM) if the endomorphism ring on $E$ over $\mathbb Q$ is larger than $\mathbb Z$.  If $E$ is a CM elliptic curve then its endomorphism ring is isomorphic to an order in an imaginary quadratic field $\mathbb Q(\sqrt{-d})$ with class number 1.  We denote this by saying that $E$ has complex multiplication by $\mathbb Q(\sqrt{-d})$.

Additionally, let $O$ denote the point at infinity (i.e. the identity).

Using these definitions, we now have the following primality test:

\begin{Def} For an elliptic curve $E$ with complex multiplication by $\mathbb Q(\sqrt{-d})$, let $P\in E(\mathbb Q)$ be a rational point of infinite order on $E$.  Moreover, let $n$ be a natural number such that $(n,6\Delta)=1$ and $\left(\frac{-d}{n}\right)=-1$.  If $n$ is prime then $$[n+1]P\equiv O\pmod n.$$
\end{Def}
In other words, for a natural number $n$, we can attempt to determine whether $n$ is prime by checking to see whether $[n+1]P\equiv O$ (mod $n$).  If it is $\not \equiv O$, $n$ is composite, while if it is $\equiv O$, we have a probable prime by Gordon's primality test.

Unfortunately, it turns out that this test, too, admits composite numbers which nevertheless satisfy the congruence above.  By analogy with the setup in section 1 above, we have the following:

\begin{Def} Let $n$ be composite.  If, for a given elliptic curve $E$, $n$ satisfies the Gordon primality test then $n$ is called an \textbf{elliptic Carmichael number for $E$} (sometimes referred to as an \textbf{$E$-elliptic Carmichael number}).  If $n$ is an $E$-elliptic Carmichael number for every CM elliptic curve, it is said to be an \textbf{elliptic Carmichael number}.
\end{Def}
As noted above, it is known that elliptic Carmichael numbers do indeed exist.  However, examples of elliptic Carmichael numbers are rather hard to come by; the smallest known elliptic Carmichael number is $$617,730,918,224,831,720,922,772,642,603,971,311.$$

\section{Introduction: Results}
Having determined that elliptic Carmichael numbers exist, it is natural to then question whether or not there are infinitely many.  Until now, this problem has only been answered conditionally; Ekstrom, Pomerance, and Thakur \cite{EPT} proved the infinitude of elliptic Carmichael numbers under the assumption of a rather strong conjecture about the size of the first prime in an arithmetic progression.

In this paper, however, we prove unconditionally that there are infinitely many elliptic Carmichael numbers.  In particular, we prove the following:

\begin{varthm1} \textit{Let $\mathcal N(x)$ denote the number of elliptic Carmichael numbers up to $x$.  Then there exists a constant $K$ for which}

\[\mathcal N(X)\gg X^{\frac{K}{(\log \log \log X)^2}}.\]
\end{varthm1}

Our task is made easier by a reformulation of the question (as noted in \cite{EPT}), as the search for elliptic Carmichael numbers can be reduced to the following Korselt-like criterion \cite{Si2}:

\begin{eccthm}  \textit{Let $n$ be a squarefree, composite positive integer with an odd number of prime factors.  Moreover, let \[\alpha=8\cdot 3\cdot 7 \cdot 11\cdot 19\cdot 43\cdot 67\cdot 163.\]
Then $n$ is an elliptic Carmichael number if for each prime $p|n$, we have $\alpha|p+1$ and $p+1|n+1$.}
\end{eccthm}
By recasting the problem with a Korselt-like condition, we can attack the problem using the tools developed for traditional Carmichael numbers.  Recent advances in techniques for finding Carmichael numbers, specifically those by Baker and Harmon \cite{BH}, Matom\"{a}ki \cite{Ma}, and the current author \cite{Wr}, have made this problem approachable, and we use the methods devised in these papers to prove the infinitude of such pseudoprimes.

%This reformulation means that the techniques of finding Carmichael numbers can be used to prove statements about elliptic Carmichael numbers as well, and

It is worth recalling that in 1994, Alford, Granville, and Pomerance \cite{AGP}, in their famous proof of infinitely many Carmichael numbers, pointed out that the Korselt's criterion lends itself to an obvious generalization for the idea of a Carmichael number:\\

\textit{One can modify our proof to show that for any fixed nonzero integer $a$, there are infinitely many squarefree, composite integers $n$ such that $p-a$ divides $n-1$ for all primes $p$ dividing $n$.  However, we have been unable to prove this for $p-a$ dividing $n-b$, for $b$ other than 0 or 1.  Such questions have siginificance for variants of pseudoprime tests, such as the Lucas probable prime test, strong Fibonacci pseudoprimes, and elliptic pseudoprimes.}\\

The current paper represents the first unconditional progress on this problem since the above statement was made.  More specifically, in the process of using the Elliptic Carmichael Condition to prove the infinitude of elliptic Carmichael numbers, we resolve the case where $a$ and $b$ are both -1.  This specific class of pseudoprimes has its own name; a square-free, composite natural number $n$ for which $p+1|n+1$ for all primes $p$ that divide $n$ is called a \textit{Lucas-Carmichael number}.  As such, we have the following:
\begin{varthm1} \textit{There are infinitely many composite, square-free natural numbers $n$ such that for any prime $p$ that divides $n$, $p+1|n+1$.  In other words, there are infinitely many Lucas-Carmichael numbers.}
\end{varthm1}

Unfortunately, our new method does not immediately extend to other $a$ and $b$.  The reasons for this are discussed in the next section.

\section{Introduction: Outline of Proof}
Much of the proof follows a similar outline to that of the previous paper by the current author \cite{Wr}, which is itself an alteration of the methods of \cite{AGP} and \cite{Ma}.  To begin, we create an $L$ with many prime factors $q$, where all of the $q$ will be -1 mod $\alpha$.  We then try to find an integer $k$ which is relatively prime to $L$ such that the set
$$\{d|L:p=dk\alpha-1\mbox{ is prime},p\mbox{ is a }\mbox{quadratic non-residue mod }q\mbox{ for all }q|L\}$$
is fairly large; note that $-1$ is quadratic non-residue for any of the chosen $q$ (since $q\equiv 3$ (mod 4)), so the two requirements on $p$ do not contradict one another.  Having found a $k$ which fits our required criteria, we use a combinatorial theorem first used in a paper of Baker and Schmidt \cite{BS} and applied to Carmichael numbers by Matom\"aki \cite{Ma} to prove that size of this set of primes is sufficient to ensure the existence of our chosen pseudoprimes.

We remark that we have very little control over $k$ apart from the requirement that it be relatively prime to $L$.  This is the reason that our methods cannot be extended to prove that there are infinitely many $n$ for which $p-a|n-b$ when $a$ and $b$ are not $\pm 1$.  More specifically, it is vitally important to our methods that all of our primes $p$ have small order modulo $k$; however, since we do not have the ability to place stringent requirements on $k$, there is no way to guarantee that $p$ would still have small order if $a$ or $b$ were a number besides $\pm 1$.

Finally, we note that one of the new ideas in this paper is to attach a parameter to our primes that can help us identify whether the number of primes being multiplied together is odd or even (see Section \ref{param}).  It is easy to see that this idea could be used to show that there are infinitely many Carmichael numbers with an odd number of factors (or, analogously, with an even number of factors).  In fact, one could expand on this idea to prove that there are infinitely many Carmichael numbers where the number of factors is congruent to $a$ (mod $M$) for any choices of $a$ and $M$.  Currently, none of these results are known; in fact, very little is known about the number of factors of a Carmichael number apart from the trivial statement that the number of factors must be greater than 2.  We take this issue up in a future paper.

%if we allowed other $a$ for our methods, $k$ can much bigger than $L$ without affecting  $b$ besides -1 because it relies heavily upon the fact that we know the the order of $-1$ is the same modulo any prime $q$, which is obviously not true for general $b$.  It also does not yet extend to other $a$ besides $-1$ because we have so little control of $k$; apart from the obvious cases of 1 or -1, we have no idea of; choosing another $a$ leaves us susceptible to the unpredictability of $k$.

%For most $a$, we can use the fact that there are infinitely many $q$ for which $a$ is a quadratic non-residue mod $q$.  Unfortunately, there are still many $a$ for which this is false (namely, those $a$ that are perfect squares).  If our $a$ happens to be a perfect square, we deal not with $a$ but with $c$, the $2^k$-th root of $a$ which is itself not a perfect square.  We note that in the case where $a$ is a perfect square, our lower bound for the density of $C^{a,-1}$ is slightly smaller, although still infinite.

\section{Primes in Arithmetic Progressions}
Let us now commence with the proof itself.

First, we need $L$ to be composed of primes $q$ such that $q-1$ is relatively smooth and -1 is a quadratic non-residue modulo each $q$.  To do this, let $1<\theta<2$, and let $P(q-1)$ be the size of the largest divisor of $q-1$.  We then define the set $\mathcal Q$ by
\[\mathcal Q=\{q\mbox{ }prime:\frac{y^\theta}{\log y}\leq q\leq y^{\theta},\mbox{ }q\equiv -1 \pmod{4\alpha},\mbox{ }P(q-1)\leq y\}.\]

We have the following lemma to tell us the size of $\mathcal Q$:

\begin{lemma}
\textit{For $\mathcal Q$ as above, there exist constants $\gamma=\gamma(\theta)$ and $Y_{\theta}$ such that
\[|\mathcal Q|\geq \gamma_\theta \frac{y^{\theta}}{\log (y^\theta)}\]
if $y>Y_{\theta}$}
\end{lemma}
\begin{proof} This proof appears in [Ma]; we replicate it here.

For $v<z$, let us denote by $\pi_{d,b}(z,v)$ the number of primes $q$ less than $z$ such that $P(q-1)\leq v$ and $q\equiv b$ (mod $d$).  Let $\frac 12<\beta<\frac 23$, and define $\epsilon=\epsilon(\beta)<\beta-\frac 12$.  Note that if $q\leq z$ is such that $q$ can be written as $q=1+q'k$ for some prime $q'\in [z^{1-\beta},z^{\frac 12-\epsilon}]$ then $P(q-1)\leq z^\beta$; each $q$ has at most two such representations.  So
\[ \pi_{d,b}(z,z^{\beta})\geq \frac 12\sum_{q'\in \mathbb P,\mbox{ }z^{1-\beta}\leq q'\leq z^{\frac 12-\epsilon}}\#\{q\mbox{ }prime,\mbox{ }\frac{z}{\log z}\leq q\leq z,\mbox{ }q\equiv 1\pmod{q'},\mbox{ }q\equiv b\pmod d\}.\]
Since $d$ is fixed and $q$ is sufficiently large relative to $d$ and $q'$, we can consolidate our requirements on $q$ to be a single congruence modulo $dq'$, and hence we may use Bombieri-Vinogradov to find that
\[ \pi_{d,b}(z,z^{\beta})\geq \sum_{q'\in \mathbb P,\mbox{ }z^{1-\beta}\leq q'\leq z^{\frac 12-\epsilon}}\frac{z}{8\phi(dq')\log z}\geq \log\left(\frac{\frac 12-\epsilon}{1-\beta}\right)\frac{z}{8\phi(d)\log z}.\]
The lemma then follows by letting $d=4\alpha$, $b=-1$, $z=y^\theta$, $\beta=min\{\frac{1}{\theta},\frac 35\}$, and $\gamma=\frac{1}{8\phi(d)}\log\left(\frac{\frac 12-\epsilon}{1-\beta}\right)$.
\end{proof}
Next, let
\[L'=\prod_{q\in \mathcal Q}q.\]
Fix $B$ such that $0<B<5/12$.  Theorem 2.1 of \cite{AGP} says that for any $x$ there exists a set of integers $\mathcal D_B(x)$, where $|\mathcal D_B(x)|$ is bounded by some constant $D_B$ and every integer in $\mathcal D_B(x)$ is greater than $\log x$, such that if $d$ is not divisible by an element in $\mathcal D_B(x)$ and $d\leq \min\{x^B,z/x^{1-B}\}$ then
\[\pi(z,d,c)\geq \frac{\pi(z)}{2\phi(d)},\]
for any $c$ with $(c,d)=1$.   As in \cite{Wr}, we wish to use this estimate and thus must be careful that our moduli are not divisible by an element of $\mathcal D_B(x)$.

To do this, let us define
\[x=\lceil \left(\alpha L'\right)^{\frac{2}{B}}\rceil.\]
For $\mathcal D_B(x)$ as described above, we can choose a set of primes $P_{B}(x)$, where  $|P_{B}(x)|\leq D_B$, such that any element in $\mathcal D_B(x)$ is divisible by at least one of the primes in $P_B(x)$.  From this, let
\[L=\prod_{q\in \mathcal Q,\mbox{ }q\not\in P_{B}(x)}q.\]
Clearly, no factor of $L$ is divisible by an element in $\mathcal D_B(x)$.

From here, we wish to show that there are a large number of primes that are congruent to $-1$ modulo a factor of $L$ as well as being quadratic non-residues modulo every prime $q$ that divides $L$.

\section{More Primes in Arithmetic Progressions}
By notation defined in \cite{Wr} (and by analogy with the convention for denoting primes in arithmetic progressions), let us use the notation $\pi(x,L,QNR)$ to indicate the number of primes up to $x$ that are quadratic non-residues modulo every divisor of $L$. Then we can prove the following:

\begin{lemma}  Let $a$ and $M$ be as above.  For $z\geq x^{1-\frac B2}$,
\[\pi(z,L,QNR)\geq \frac{z}{2^{\omega(L)+1}\log z},\]
where $\omega(L)$ denotes the number of prime factors of $L$.
\end{lemma}

\begin{proof}
First, note that $L\leq \min\{x^{B},z/x^{1-B}\}$, which means that we can apply Theorem 2.1 from \cite{AGP} as described above.

Now, as we are working mod $L$, we note that the number of congruence classes that are quadratic non-residues modulo each prime $q|L$ is exactly $\frac{q-1}{2}$ of the $q-1$ classes which can contain prime numbers.  By Chinese Remainder Theorem, this means that the number of congruence classes mod $L$ that are quadratic non-residues for every $q$ is exactly $\prod_{q|L}(\frac{q-1}{2})$ of the $\prod_{q|L}q-1=\phi(L)$ congruence classes which yield a prime.

%the number of congruence classes mod $L\cdot m$ is $\leq m(\log\log\log x)^{\log\log\log x}$, which, for large $x$, is clearly less than $x^B$ for any positive, fixed $B$.  Since all of the prime factors of $L\cdot m$ are clearly less than $\log x$, it follows that

%
Obviously, there are $\phi(L)$ congruence classes mod $L$ which can contain a prime.  Let $u$ be a representative of such a class mod $L$.  For any such $u$, the number of primes congruent to $u$ mod $L$ is

\[\geq \frac{z}{2\phi(L)\log z}.\]
From above, for each $q|L$, there are $\prod_{q|L}(\frac{q-1}{2})$ of these classes which would yield a quadratic non-residue mod every divisor of $L$.  So the number of primes in the required congruence classes is
\begin{align*}
\pi(z,L,QNR)\geq &\frac{z\left(\prod_{q|L}(\frac{q-1}{2})\right)}{2\phi(L)\log z}\\
=&\frac{z\left(\frac{\phi(L)}{2^{\omega(L)}}\right)}{2\phi(L)\log z}\\
=&\frac{z}{2^{\omega(L)+1}\log z}
\end{align*}
which is as required.\end{proof}

Now, for a given integer $d|L$ with $1\leq d\leq x^B$ for some fixed $B>0$, we wish to count the number of primes $p$ for which $p\equiv -1$ (mod $d$) and $((p+1)/d,L)=1$.  However, since we chose $x$ to be $\lceil \left(\alpha L'\right)^{\frac{2}{B}}\rceil$, every divisor $d$ of $L$ will be $\leq x^B$.  To this end, we find the following:

\begin{lemma} \textit{Let $B<5/12$, and let $L$ be as above.  Then there exists a $k\leq x^{1-\frac B2}$ with $(k,L)=1$ such that}
\begin{align*}
\#\{d|L &:p=dk\alpha-1\mbox{ }is\mbox{ }prime,\mbox{ }p\mbox{ }is\mbox{ }a\mbox{ }QNR\mbox{ }mod\mbox{ }q\mbox{ }for\mbox{ }every\mbox{ }q|L,\mbox{ }p\leq x\}\\
&\geq \frac{\left(\frac 32\right)^{\omega(L)}}{4\phi(\alpha)\log x}.
\end{align*}
%\#\{d|L &:p=dk+1\mbox{ }is\mbox{ }prime,\mbox{ }p\mbox{ }is\mbox{ }a\mbox{ }QR\mbox{ }mod\mbox{ }L,\mbox{ }p\equiv a\pmod M,\mbox{ }p\leq x\}\\
%&\geq \frac{1}{4(2^{\omega(L)}\phi(M)\phi(d)\log x)}\#\{d|L:1\leq d\leq x^B\}.

\end{lemma}
\begin{proof}
Above, we proved that for $z\geq x^{1-\frac B2}$,
\[\pi(z,L,QNR)\geq \frac{z}{2(2^{\omega(L)})\log z}.\]
If we add the additional constraint that the prime $p$ also be $-1$ mod $d\alpha$ for a given $d|L$, we have
\[\pi(dx^{1-\frac B2},d\alpha,-1)\cap \pi(dx^{1-\frac B2},L,QNR)\geq \frac{dx^{1-\frac B2}}{2\cdot 2^{\omega(L)-\omega(d)}\phi(d\alpha)\log x},\]
where the savings of $2^{\omega(d)}$ in the denominator comes from the fact that we no longer have to worry whether $p$ is a quadratic non-residue mod $d$ (since -1 is a quadratic non-residue mod all prime divisors of $L$), and hence the requirement that $p$ be a quadratic non-residue mod $L$ is satisfied if $p$ is a quadratic non-residue mod $\frac Ld$.

In order to eliminate the possibility of double-counting the primes in our set, we must now determine how many of these primes satisfy the additional condition of $((p+1)/d,L)=1$.  We require the technical condition that $\sum_{q|L}\frac{1}{q-1}\leq \frac{1}{64}$; however, just as in [AGP], this is easily verified for the $L$ we have chosen.

Now, for any prime $q$ which divides $L$, we have (by Montgomery and Vaughan's explicit version of the Brun-Titchmarsh theorem [MV]) that
\begin{align*}
\pi(dx^{1-\frac B2},d\alpha q,-1)&\cap \pi(dx^{1-\frac B2},L,QNR)\\
\leq & \frac{2dx^{1-\frac B2}}{2^{\omega(L)-\omega(d)-1}\phi(dq)\log (x^{1-\frac B2}/q)}\\
\leq & \frac{8}{(q-1)(1-\frac B2)}\frac{dx^{1-\frac B2}}{2^{\omega(L)-\omega(d)}\phi(d\alpha)\log x},
\end{align*}
since $q\leq x^{(1-\frac B2)/2}$ by construction.  Putting these together, we have that
\begin{align*}
\pi(dx^{1-\frac B2},d\alpha,-1)&\cap \pi(dx^{1-\frac B2},L,QNR)\\
&-\sum_{q|L,\mbox{ }q\mbox{ }prime}\pi(dx^{1-\frac B2},d\alpha q,-1)\cap \pi(dx^{1-\frac B2},L,QNR)\\
\geq &\frac{dx^{1-\frac B2}}{2\cdot 2^{\omega(L)-\omega(d)}\phi(d\alpha)\log x}-\sum_{q|L,\mbox{ }q\mbox{ }prime}\frac{8}{(q-1)(1-\frac B2)}\frac{dx^{1-\frac B2}}{2^{\omega(L)-\omega(d)}\phi(d\alpha)\log x}\\
\geq &\frac{x^{1-\frac B2}}{4\cdot 2^{\omega(L)-\omega(d)}\phi(\alpha)\log x},
\end{align*}
where the final step is by the fact that $\sum_{q|L}\frac{1}{q-1}\leq \frac{1}{64}$.  This means that we have at least
\[\sum_{d|L}\frac{x^{1-\frac B2}}{4\cdot 2^{\omega(L)-\omega(d)}\phi(\alpha)\log x}\]
pairs $(p,d)$ such $p$ is prime, $d|L$, $(p+1)/d=k$ is relatively prime to $L$, $p$ is a quadratic non-residue mod $L$), $p\leq dx^{1-\frac B2}$, and $d\leq x^\frac B2$.

Since the number of possible distinct values for $(p+1)/d$ is bounded by $x^{1-\frac B2}$, there must exist some value $k$ such that $(k,L)=1$ and $k$ has at least
\[\sum_{d|L}\frac{2^{\omega(d)}}{4\cdot 2^{\omega(L)}\phi(\alpha)\log x}\]
representations as $(p+1)/d$ for $p,d$ as above.

The numerator can be evaluated by a standard combinatorial identity:
\[\sum_{d|L}2^{\omega(d)}=\sum_{i=0}^{\omega(L)}\left(\begin{array}{c} \omega(L) \\ i \end{array} \right) 2^{\omega(L)-i}=(2+1)^{\omega(L)}=3^{\omega(L)},\]
and hence
\[\sum_{d|L}\frac{2^{\omega(d)}}{4\cdot 2^{\omega(L)}\log x}=\frac{\left(\frac 32\right)^{\omega(L)}}{4\phi(\alpha)\log x}.\label{eqn1}\]\end{proof}

%\begin{corollary} \textit{Let $B<5/12$, and let $L$ be as above.  Then there exists a $k\leq x^{1-\frac B2}$ with $(k,L)=1$ such that}
%\begin{align*}
%\#\{d|L &:p=dk\alpha-1\mbox{ }is\mbox{ }prime,\mbox{ }p\mbox{ }is\mbox{ }a\mbox{ }QNR\mbox{ }mod\mbox{ }q\mbox{ }for\mbox{ }every\mbox{ }q|L,\mbox{ }p\leq x\}\\
%&\geq \frac{\left(\frac 32\right)^{\omega(L)}}{4\phi(\alpha)\log x}.
%\end{align*}
%\#\{d|L &:p=dk+1\mbox{ }is\mbox{ }prime,\mbox{ }p\mbox{ }is\mbox{ }a\mbox{ }QR\mbox{ }mod\mbox{ }L,\mbox{ }p\equiv a\pmod M,\mbox{ }p\leq x\}\\
%&\geq \frac{1}{4(2^{\omega(L)}\phi(M)\phi(d)\log x)}\#\{d|L:1\leq d\leq x^B\}.

%\end{corollary}

%\begin{proof}
%The proof is the same as the previous lemma, except that instead of requiring $p$ to be $-1$ mod $d$, we would require it to be $-1$ mod $d\alpha$ in each step.  Doing so changes the $\phi(d)$ to $\phi(d\alpha)$, which is the same as $\phi(d)\phi(\alpha)$ since $d$ and $\alpha$ are coprime.
%\end{proof}

Let $k_0$ be the $k$ found by the previous lemma.  We define
\[\mathcal P=\{d|L:p=dk_0\alpha-1,\mbox{ }is\mbox{ }prime,\mbox{ }p\mbox{ }is\mbox{ }a\mbox{ }QNR\mbox{ }mod\mbox{ }q\mbox{ for every }q|L,\mbox{ }p\leq x\}.\]
The pseudoprimes will be generated from products of primes in this set.

\section{Sizes of Other Important Quantities}
Here, we introduce a theorem of Baker and Schmidt [\cite{BS}, Proposition 1] that was reinterpreted in terms of pseudoprimes by Matom\"aki \cite{Ma}.  To begin, for an abelian group $G$, $n(G)$ is defined to be the smallest number such that a collection of at least $n(G)$ elements must contain some subset whose product is the identity.  Based upon the work of van Emde Boas and Kruyswijk \cite{EK} and Meshulam \cite{Me}, it is known that
\[n(G)\leq \lambda(G)\left(1+\frac{\log|G|}{\lambda(G)}\right).\]
With this notation, we may now state the theorem.
\begin{theorem}\label{thm}
\textit{For any multiplicative abelian group $G$, write}
\[s(G) = \lceil 5\lambda(G)^2 \Omega (\lambda(G)) \log(3\lambda(G)
\Omega (|G|))\rceil,\]
\textit{where $\Omega(|G|)$ indicates the number of prime divisors (up to multiplicity) of $|G|$.}

\textit{Let $A$ be a sequence of length $n$ consisting of non-identity elements of $G$. Then there exists a non-trivial subgroup $H\subset G$ such that}\\

\textit{(i) If $n\geq s(G)$, then, for every $h\in H$, $A\cap H$ has a subsequence whose product is $h$.}

\textit{(ii) If $t$ is an integer such that $s(G) < t < n-n(G)$ then, for every $h \in H$, $A$ has at least} $\left(\begin{array}{c} n-n(G)\\ t-n(G)\end{array}\right)/\left(\begin{array}{c} n\\n(G)\end{array}\right)$
\textit{distinct subsequences of length at most $t$ and at least
$t - n(G)$ whose product is $h$.}
\end{theorem}
\begin{proof}
See Lemma 6 of \cite{Ma}.

\end{proof}
For the rest of the paper, we will use the result from Theorem \ref{thm} in the case where
\[G=(\mathbb Z/L\mathbb Z)^\times \times \{-1,1\},\]
where $\{-1,1\}$ is taken to be a group with multiplication.  We will express any element of $G$ as an ordered pair $(a,b)$, where $a\in (\mathbb Z/L\mathbb Z)^\times$ and $b\in\{-1,1\}$, and multiplication of elements of $G$ will be performed componentwise (i.e. $(a,b)\times (c,d)=(a\times c,b\times d)$).

In order to show that Theorem \ref{thm} can be invoked, we must bound the quantities $n$ and $s$ in terms of $y$ and $\theta$:
\begin{lemma}\label{lemfour2}  For $G$, $s(G)$, and $n(G)$ as defined above and $y$ sufficiently large,
\[s(G)\leq e^{7\theta y},\]
\[n(G)\leq e^{3\theta y}.\]
\end{lemma}
\begin{proof}
Follows from the fact that $\lambda(G)\leq e^{2\theta y}$ and $M$ is a constant.
\end{proof}
We also require an estimate for the size of $L$ itself:
\begin{lemma}\label{lemfour}
\textit{For $y$ sufficiently large, there exist constants $0<\kappa_1<\kappa_2$ such that}
\[e^{\kappa_1 y^\theta}\leq L\leq e^{\kappa_2 y^\theta}.\]
\end{lemma}
\begin{proof}
For the upper bound, it is well known that the product of the primes up to $x$ is less than $e^{1.02 x}$ (see \cite{RS}).

For the lower bound, define $E_\mathcal Q(n)$ to be 1 if $n\in \mathcal Q-D_B(x)$ and 0 otherwise.  Then
\[\log L=\sum_{n=\frac{y^\theta}{\log y}}^{y^\theta}E_\mathcal Q(n)\log n\geq \left|\mathcal Q-D_B(x)\right|\log\left(\frac{y^\theta}{\log y}\right)\geq \frac{\gamma}{2}y^\theta.\]
\end{proof}

\section{Proof of Theorem}\label{param}
From this, we may now prove the infinitude of our elliptic Carmichael numbers.  Let
\[j=\prod_{q|L}\frac{q-1}{2}.\]
Note that $j$ is odd, since all of the $q$ are 3 (mod 4).

\begin{theorem}
\textit{Let $H$ be as defined in Theorem \ref{thm}.  There exists an element $h\in H$ such that
\[h=(-1,-1)\]
Equivalently, some subset of the primes in $\mathcal P$ multiply to a number $n$ for which $p|n$ implies $p+1|n+1$.}
\end{theorem}

\begin{proof}
First, define
$$A=\{(p,-1):\mbox{ }p\in\mathcal P\}.$$
Clearly, $|A|>s(G).$  So $A\cap H$ is non-empty.  Thus, let $p_H$ be a prime such that $(p_H,-1)\in A\cap H$.  Consider
\[h=(p_H,-1)^j=(p_H^j,(-1)^j)\]
Since $p_H$ is a quadratic non-residue modulo each prime $q$ that divides $L$, we have
\[p_H^j\equiv (p_H^{\frac{q-1}{2}})^\frac{j}{\left(\frac{q-1}{2}\right)}\equiv (-1)^\frac{j}{\left(\frac{q-1}{2}\right)}\equiv -1\pmod q\]
for each of these $q$ (since $j$ is odd).  Moreover, since $j$ is odd, $(-1)^j=-1$.  Thus
\[h=(-1,-1).\]
For the second half of the theorem, let us consider a subset $\{p_1,...,p_s\}$ of $\mathcal P$ such that
\[(p_1,-1)\times (p_2,-1)...\times (p_s,-1)=h.\]
Clearly, such a subset of $\mathcal P$ exists by Theorem \ref{thm}.  Let
\[m=p_1...p_s.\]
We note first that $s$ must be odd, since $(-1)^s=-1$ (as a result of the fact that the latter coordinate in $h$ is -1).  As the primes in $\mathcal P$ are all of the form $dk\alpha-1$ (i.e. $-1$ (mod $k$)), this means that our $m$ must be -1 (mod $k$) as well.  Moreover, examining $m$ modulo $L$ yields
$$m \equiv p_1p_2...p_s \equiv -1 \pmod L.$$
Thus, for any prime divisor $p_i$ of $m$, $p_i+1=dk|Lk|m+1$.  Moreover, by our construction of $\mathcal P$, we know that $\alpha|p_i+1$.  So $m$ is an elliptic Carmichael number.
\end{proof}

Let us now find the lower bound on the asymptotic.  We will proceed similarly to the paper of \cite{Wr}.  First, we must choose a $t$ such that $s(G)<t < n -n(G)$.  As such, we will let
\[t=\frac{\left(\frac 65\right)^{\omega(L)}}{60\cdot \phi(\alpha)\log x}.\]
\begin{varthm1}
Let $\mathcal N(X)$ be the number of elliptic Carmichael numbers up to $X$.  Then there exists a constant $K>0$ such that
\[\mathcal N(X)\gg \left(X\right)^\frac{K}{(\log\log\log X)^2}.\]
\end{varthm1}
\begin{proof}
Since we chose $t$ such that $s(G)<t<|\mathcal P|-n(G)$, we know that the number of products of at most $t$ primes in $|\mathcal P|$ whose product is -1 mod $L$ is at least
\[\left(\begin{array}{c} |\mathcal P|-n(G)\\ t-n(G)\end{array}\right)/\left(\begin{array}{c} |\mathcal P|\\n(G)\end{array}\right).\]
From Lemma \ref{lemfour2}, $n(G)$ is much smaller than both $\mathcal P$ and $t$.  So we can say that that
\[|\mathcal P|-n(G)\geq \frac 45|\mathcal P|,\label{eq1}\]
and
\[t-n(G)\geq \frac 23 t.\label{eq2}\]
We will also use the standard bound that
\[\left(\frac{u}{v}\right)^v\leq \left(\begin{array}{c} u\\v\end{array}\right)\leq \left(\frac{ue}{v}\right)^v.\label{eq3}\]
So
\begin{align*}
\left(\begin{array}{c} |\mathcal P|-n(G)\\ t-n(G)\end{array}\right)/& \left(\begin{array}{c} |\mathcal P|\\n(G)\end{array}\right)\\
\geq & \left(\begin{array}{c} |\mathcal P|-n(G)\\ t-n(G)\end{array}\right)/ \left(\begin{array}{c} |\mathcal P|\\ \frac t3\end{array}\right)\\
\geq &\left(\frac{\frac 45 |\mathcal P|}{t}\right)^{\frac 23 t}/\left(\frac{3e|\mathcal P|}{t}\right)^{\frac t3}\\
\geq &\left(\frac{16}{75e}\right)^{\frac 13 t}\left(\frac{|\mathcal P|}{t}\right)^{\frac 13 t}\\
\end{align*}
By the calculation of $|\mathcal P|$ and the definition of $t$, we have
\begin{align*}
\left(\frac{16}{75e}\right)^{\frac 13 t}&\left(\frac{|\mathcal P|}{t}\right)^{\frac 13 t}\\
\geq &\left(\frac{16\cdot 15}{75e}\right)^{\frac 13 t}\left(\frac{\left(\frac 32\right)^{\omega(L)}}{\left(\frac 65\right)^{\omega(L)}}\right)^{\frac 13 t}\\
\geq &\left(\left(\frac 54\right)^{\omega(L)}\right)^{\frac 13 t}.\\
\end{align*}
Since we chose $L$ to be the product of distinct primes, it is easy to give a lower bound for $\omega{L}$:
\[\omega(L)\geq \gamma\frac{y^{\theta}}{\log y^\theta}.\]
Define
\[X=x^t.\]
Using the bound for $\omega(L)$ above and the definition of $t$, we have that for $y$ sufficiently large,
\[(\log\log\log X)^2\geq \log y^\theta.\]
Moreover, from Lemma \ref{lemfour}, we have
\[X\leq (e^{\frac{2\kappa_2}{B} y^\theta})^{2t}.\]
Note also that any Carmichael number that is the product of at most $t$ primes in $\mathcal P$ must be less than $X$.  So the number derived from the combinatorial functions above is a lower bound for $\mathcal N(X)$.

As such, we have that, for sufficiently large values of $y$,
\begin{align*}
\mathcal N(X)\geq &\left(\left(\frac 54\right)^{\frac 13\omega(L)}\right)^{t}\\
\geq & \left(\left(e^{\frac{\kappa_1 y^\theta}{\log y^\theta}}\right)^t\right)^{\frac{\gamma\log(1.2)}{3}}\\
\geq & \left(X\right)^\frac{\gamma\log(1.2)B\kappa_1}{12\kappa_2(\log\log\log X)^2},
\end{align*}
where the second line is again from Lemma \ref{lemfour}.  The theorem then follows.
\end{proof}

\bibliographystyle{amsalpha}

\end{document}